\documentclass[a4paper,11pt]{article}

\usepackage[defaultlines=5,all]{nowidow}		
\usepackage{indentfirst}
\usepackage[T1]{fontenc}
\usepackage[utf8]{inputenc}
\usepackage{authblk}
\usepackage{diagbox}
\usepackage{amsthm,amsmath}
\usepackage{tikz}
\usepackage{algorithm}
\usepackage{algorithmic}	
\usepackage{paralist}
\usepackage[shortlabels]{enumitem}
\usepackage{subcaption}
\usepackage{hyperref}
\usepackage{multirow}
\usepackage{setspace}		
\usepackage{stfloats}
\usepackage{verbatim}
\hypersetup{
	colorlinks=true,
	citecolor = blue,
	linkcolor=blue,
	filecolor=magenta,
	urlcolor=cyan,
}

\newtheorem{theorem}{Theorem}
\newtheorem{lemma}{Lemma}
\newtheorem{claim}{Claim}

\title{{\bf  Some Results on Critical ($P_5,H$)-free Graphs}}

\author[a,b]{Wen Xia}
\author[c]{Jorik Jooken}
\author[c,d]{Jan Goedgebeur}
\author[,a,b]{Shenwei Huang\thanks{Email: shenweihuang@nankai.edu.cn.}}

\affil[a]{College of Computer Science, Nankai University, Tianjin 300071, China}
\affil[b]{Tianjin Key Laboratory of Network and Data Security Technology, Nankai University, Tianjin 300071, China}
\affil[c]{Department of Computer Science, KU Leuven Campus Kulak-Kortrijk, 8500 Kortrijk, Belgium}
\affil[d]{Department of Applied Mathematics, Computer Science and Statistics, Ghent University, 9000 Ghent, Belgium}

\date{}

\begin{document}
	
	\maketitle
\begin{abstract}

 Given two graphs $H_1$ and $H_2$, a graph is $(H_1,H_2)$-free if it contains
no induced subgraph isomorphic to $H_1$ nor $H_2$. A graph $G$ is $k$-vertex-critical if  every
 proper induced subgraph of $G$ has chromatic number less than $k$, but $G$ has chromatic number $k$.
 The study of $k$-vertex-critical graphs for specific graph classes is an important topic in algorithmic graph theory
 because if the number of such graphs that are in a given hereditary graph class is finite, then there exists
 a polynomial-time certifying algorithm to decide
the $k$-colorability of a graph in the class.

In this paper, we show that:
(1) for $k \ge 1$, there are finitely many $k$-vertex-critical
$(P_5,K_{1,4}+P_1)$-free graphs;
(2) for $s \ge 1$,  there are finitely
many 5-vertex-critical $(P_5,K_{1,s}+P_1)$-free graphs;
(3) for $k \ge 1$, there are finitely many $k$-vertex-critical $(P_5,\overline{K_3+2P_1})$-free graphs.
Moreover, we characterize all $5$-vertex-critical
$(P_5,H)$-free graphs where $H \in \{K_{1,3}+P_1,K_{1,4}+P_1,\overline{K_3+2P_1}\}$
using an exhaustive graph generation algorithm.

{\bf Keywords.} Graph coloring; $k$-vertex-critical graphs; Strong perfect graph theorem; Polynomial-time certifying algorithms.
	
\end{abstract}	
\section{Introduction}

All graphs in this paper are finite, undirected and simple. Let $K_n$ be the complete graph on $n$ vertices. Let $P_t$ and $C_t$ denote the path and the cycle on $t$ vertices, respectively. The complement of $G$ is denoted by $\overline{G}$. For two graphs
$G$ and $H$, we use $G+H$ to denote the disjoint union of $G$ and $H$. For a positive integer $r$, we use $rG$
to denote the disjoint union of $r$ copies of $G$. For $s,r \geq 1$, let $K_{r,s}$ be the complete bipartite
graph with one part of size $r$ and the other part of size $s$. We use $\text{d}(u, v)$ to denote the \textit{distance} between vertices $u$
and $v$, i.e., the number of edges of a shortest path between vertices $u$ and $v$.

A \textit{$k$-coloring}
of a graph $G$ is an assignment of $k$ different colors to the vertices of
$G$ such that adjacent vertices receive different colors. The minimum $k$ for
which $G$ has a $k$-coloring is called the \textit{chromatic number} of $G$ and is
denoted by $\chi(G)$. A graph $G$ is \textit{$k$-chromatic} if $\chi(G)= k$.
A graph $G$ is \textit{$k$-critical} if it is $k$-chromatic and $\chi(G-e) < \chi(G)$ for
any edge $e \in E(G)$.
A graph is \textit{critical} if it is $k$-critical  for some
integer $k \ge 1$. For instance, the class of 3-critical graphs is the family
of all chordless odd cycles. Vertex-criticality is a weaker notion. A graph $G$ is \textit{$k$-vertex-critical}
if $\chi(G)=k$ and $\chi(G-v) < k$ for any $v \in V(G)$.

For a set $\mathcal{H}$ of graphs, a graph $G$ is \textit{$\mathcal{H}$-free} if
it does not contain any member of $\mathcal{H}$ as an induced subgraph.
When $\mathcal{H}$ consists of a single graph $H$ or two graphs $H_1$ and
$H_2$, we write $H$-free and $(H_1,H_2)$-free instead of $\{H\}$-free and $\{H_1,H_2\}$-free, respectively.
We say that $G$ is \textit{$k$-vertex-critical $\mathcal{H}$-free} if it is $k$-vertex-critical
and $\mathcal{H}$-free. We say that $G$ is \textit{$k$-critical $\mathcal{H}$-free} if $G$ is $\mathcal{H}$-free, $\chi(G)=k$ and for every
$\mathcal{H}$-free proper subgraph $G'$ of $G$ we have $\chi(G') \le k-1$.

 The problem of deciding if a graph is $k$-colorable is a well-known NP-complete problem. However, the picture changes if one restricts the structure of
 the graphs under consideration. For example, $k$-colorability can be decided in polynomial time for all $k$ for perfect graphs~\cite{GLS84} and $P_5$-free
 graphs~\cite{HKLSS10}. In fact, deciding $k$-colorability of $H$-free graphs remains NP-complete if $H$ contains an induced claw~\cite{H81,LG83} or cycle~\cite{KL07,MP96}.
 Moreover, the problem remains NP-complete  for $P_6$-free graphs for all $k \ge 5$ and for $P_7$-free graphs for all $k \ge 4$~\cite{H16}.
 Therefore, $P_5$ is the largest connected graph whose forbidding results in the existence of polynomial-time
 $k$-colorability algorithms for all $k$ (assuming P $\neq$ NP). However, the algorithm from~\cite{HKLSS10} produces a $k$-coloring if one exists,
 but does not produce an easily verifiable certificate when such coloring does not exist.

An algorithm is \textit{certifying} if it returns with each output a simple and easily verifiable certificate
that the particular output is correct. For a $k$-colorability algorithm to be certifying, it
should return a $k$-coloring if one exists and a $(k+1)$-vertex-critical induced subgraph if
the graph is not $k$-colorable.

The classification of all 12 4-vertex-critical $P_5$-free graphs was used to develop
a linear-time certifying 3-colorability algorithm for $P_5$-free graphs~\cite{BHS09,MM12}.
In contrast, there are infinitely
many $(k+1)$-vertex-critical $P_5$-free graphs for $k \ge 4$.
This led to significant interest in determining for which
subfamilies of $P_5$-free graphs the polynomial-time $k$-colorability algorithms can be extended to
certifying ones. In 2021, Cameron, Goedgebeur, Huang and Shi~\cite{CGHS21} obtained the following dichotomy result.

\begin{theorem}[\cite{CGHS21}]
		Let $H$ be a graph of order 4 and $k\ge 5$ be a fixed integer. Then there are infinitely many k-vertex-critical $(P_5,H)$-free graphs if and only if $H$ is $2P_2$ or $P_1+K_3$.
	\end{theorem}

This theorem completely solves the finiteness problem of $k$-vertex-critical $(P_5,H)$-free graphs for $|H|=4$. In~\cite{CGHS21}, the authors also posed the
natural question of which five-vertex graphs $H$ lead to finitely many $k$-vertex-critical $(P_5,H)$-free graphs.
Already quite some progress has been made towards solving this question. In particular, it is known that there are only finitely many $k$-vertex critical $(P_5,H)$-free
for $k \le 5$ when $H=C_5$~\cite{HMRSV15}, bull~\cite{HLX23} or chair~\cite{HL23} and for all $k$ when:
$H$ = dart~\cite{XJGH23}, banner~\cite{BGS22}, $P_2+3P_1$~\cite{CHS22}, $P_3+2P_1$~\cite{ACHS22},
            $ \overline{P_3+P_2}$ and gem~\cite{CGS}, $ K_{2,3}$ and $K_{1,4}$~\cite{KP17}, and $\overline{P_5}$~\cite{DHHMMP17}.
In contrast, there are infinitely many $k$-vertex-critical $(P_5,H)$-free graphs for $k > 5$ when
$H = 2K_2+P_1$, $P_3+P_2$, $K_3+2P_1$, co-dart, $diamond+ P_1$, $P_5$, $K_3+P_2$,
co-banner, butterfly, $C_5$, kite, $K_4+P_1$ and $\overline{claw+P_1}$. We refer the interested reader to~\cite{CH23} for more details.

\noindent {\bf Our contributions.} In this paper, we continue to study the finiteness of vertex-critical $(P_5,H)$-free graphs. Our main results are as follows:

\begin{theorem}\label{th:k-vertex-critical}
   For every fixed integer $k \ge 1$, there are finitely many $k$-vertex-critical $(P_5, K_{1,4}+P_1)$-free graphs.
\end{theorem}
\begin{theorem}\label{th:5-vertex-critical}
For every fixed integer $s \ge 1$, there are finitely many $5$-vertex-critical $(P_5, K_{1,s}+P_1)$-free graphs.
\end{theorem}
\begin{theorem}\label{th:k3+2p1}
For every fixed integer $k \ge 1$, there are finitely many $k$-vertex-critical $(P_5, \overline{K_3+2P_1})$-free graphs.
\end{theorem}

We prove two Ramsey-type statements (see \autoref{ST} and \autoref{lem:k3+2p1}) which allow us to prove \autoref{th:k-vertex-critical}
 and \autoref{th:k3+2p1} by induction on $k$. In particular, we perform a careful structural analysis via the Strong Perfect Graph Theorem combined with the pigeonhole principle
and Ramsey's Theorem based on the properties of vertex-critical graphs. Our results imply the existence of a polynomial-time certifying algorithm to decide the $k$-colorability of
$(P_5,H)$-free graphs for $k \ge 1$ and $H = \overline{K_3+2P_1}$ or $K_{1,4} +P_1$ and a  polynomial-time certifying algorithm to decide the 4-colorability of $(P_5, K_{1,s}+P_1)$-free graphs.
Moreover, our results affirmatively answer the finiteness problem posed in~\cite{CGHS21} for $H \in \{K_{1,3}+P_1, \overline{K_3+2P_1}\}$. This only leaves the cases of  $H \in \{P_4+P_1,\overline{diamond+P_1},C_4+P_1,\overline{P_3+2P_1},W_4\}$ for $k \ge 5$, $H \in \{chair,bull,K_5\}$ for $k \ge 6$ and $H = K_5 - e$ for $k=5,6,7$ as open problems.

The remainder of the paper is organized as follows. We present some preliminaries in Section~\ref{Preliminarlies}. We show that there are finitely many $k$-vertex-critical ($P_5$, $K_{1,4}+P_1$)-free graphs for all $k \ge 1$ in Section~\ref{sec:k1,4}. In Section~\ref{sec:k1,s}, we prove that there are finitely many 5-vertex-critical $(P_5,K_{1,s})$-free graphs. Next, we show that there are finitely many $k$-vertex-critical $(P_5,\overline{K_3+2P_1})$-free graphs for all $k \ge 1$ in Section~\ref{sec:k3+2p1}. Finally, in Section~\ref{sec:characterization} we characterize all $5$-vertex-critical $(P_5,H)$-free graphs for $H \in \{K_{1,3}+P_1,K_{1,4}+P_1,\overline{K_3+2P_1}\}$ using a specialised graph generation algorithm.

\section{Preliminaries}\label{Preliminarlies}
For general graph theory notation we follow~\cite{BM08}. For $k\ge 4$, an induced  cycle of length $k$ is called a {\em $k$-hole}. A $k$-hole is an {\em odd hole} (respectively {\em even hole}) if $k$ is odd (respectively even). A {\em $k$-antihole} is the complement of a $k$-hole. Odd and even antiholes are defined analogously.
	
	Let $G=(V,E)$ be a graph. If $uv\in E(G)$, we say that $u$ and $v$ are {\em neighbors} or {\em adjacent}, otherwise $u$ and $v$ are {\em nonneighbors} or  {\em nonadjacent}. The {\em neighborhood} of a vertex $v$, denoted by $N_G(v)$, is the set of neighbors of $v$. For a set $X\subseteq V(G)$, let $N_G(X)=\cup_{v\in X}N_G(v)\setminus X$. We shall omit the subscript whenever the context is clear. For $x \in V(G)$ and $S\subseteq V(G)$, we denote by $N_S(x)$ the set of neighbors of $x$ that are in $S$, i.e., $N_S(x)=N_G(x)\cap S$. For two sets $X,S\subseteq V(G)$, let $N_S(X)=\cup_{v\in X}N_S(v)\setminus X$.

For $X,Y\subseteq V(G)$, we say that $X$ is {\em complete} (resp. {\em anticomplete}) to $Y$ if every vertex in $X$ is adjacent (resp. nonadjacent) to every vertex in $Y$. If $X=\{x\}$, we write ``$x$ is complete (resp. anticomplete) to $Y$'' instead of ``$\{x\}$ is complete (resp. anticomplete) to $Y$''. If a vertex $v$ is neither complete nor anticomplete to a set $S$, we say that $v$ is {\em mixed} on $S$. For a vertex $v\in V$ and an edge $xy\in E$, if $v$ is mixed on $\{x,y\}$, we say that $v$ is {\em mixed} on $xy$. For a set $H\subseteq V(G)$, if no vertex in $V(G) \setminus H$ is mixed on $H$, we say that $H$ is a {\em homogeneous set}, otherwise $H$ is a {\em nonhomogeneous set}. We say that a vertex $w$ {\em distinguishes} two vertices $u$ and $v$ if $w$ is adjacent to exactly one of $u$ and $v$.

 A vertex subset $S\subseteq V(G)$ is {\em stable} if no two vertices in $S$ are adjacent. A {\em clique} is the complement of a stable set.
 The clique number of $G$, denoted by $\omega(G)$, is the size of a largest clique in $G$.
 Two nonadjacent vertices $u$ and $v$ are said to be {\em comparable} if $N(v)\subseteq N(u)$ or $N(u)\subseteq N(v)$.
 For an induced subgraph $A$ of $G$, we write $G-A$ instead of $G-V(A)$. For $S\subseteq V$, the subgraph \emph{induced} by $S$ is denoted by $G[S]$. We use $R(s,t)$ to denote the Ramsey number where $s$ and $t$ are two parameters, namely the minimum positive integer $n$ such that every graph of order
 $n$ contains either a stable set of size $s$ or a clique of size $t$.

 Let $G=(V,E)$ be a graph and $H$ be an induced subgraph of $G$.
We partition $V\setminus V(H)$ into subsets with respect to $H$ as follows:
for any $X\subseteq V(H)$, we denote by $S(X)$ the set of vertices
in $V\setminus V(H)$ that have $X$ as their neighborhood among $V(H)$, i.e.,
\[S(X)=\{v\in V\setminus V(H): N_{V(H)}(v)=X\}.\]
For $0\le m\le |V(H)|$, we denote by $S_m$ the set of vertices in $V\setminus V(H)$ that have exactly $m$
neighbors in $V(H)$. Note that $S_m=\bigcup_{X\subseteq V(H): |X|=m}S(X)$.

 We proceed with a few useful results that will be used later. The first folklore property of vertex-critical graph is that such graphs
 contain no comparable vertices. A generalization of this property was presented in~\cite{CGHS21}.

     \begin{lemma}[\cite{CGHS21}] \label{lem:XY}
		Let $G$ be a $k$-vertex-critical graph. Then $G$ has no two nonempty disjoint subsets $X$ and $Y$ of $V(G)$ that satisfy all the following conditions.
		\begin{itemize}
			\item $X$ and $Y$ are anticomplete to each other.
			\item $\chi(G[X])\le\chi(G[Y])$.
			\item Y is complete to $N(X)$.
		\end{itemize}
	\end{lemma}

\begin{lemma}[\cite{XJGH23}]\label{lem:homoge}
Let $G$ be a $k$-vertex-critical graph and $S$ be a homogeneous set of $V(G)$.
For each component $A$ of $G[S]$, if $\chi(A) = m$ with $m < k$, then $A$ is an $m$-vertex-critical graph.
\end{lemma}

The following two lemmas are two examples of Ramsey-type arguments which say
that if a vertex subset in a graph has large size, then it has large chromatic
number.

\begin{lemma}\label{ST}
Let $G$ be a $(K_{1,s}+P_1)$-free graph where $s \ge 1$ and $c,d$ be two fixed integers. Let $S,T$ be two disjoint subsets of $V(G)$ such that
$|S| \le c$, $\chi(G[T]) \le d$ and $\overline{G[T \cup S]}$ is connected. If for any $w \in S$ there exists
a vertex $w' \in V(G)\setminus{S\cup T}$ such that $ww' \notin E(G)$ and $w'$ is complete to $T$, then $T$ is bounded.
\end{lemma}

\begin{proof}
       Let $N_0 = S$, $N_i = \{v|v \in T\setminus {\cup_{j=0}^{i-1}N_j}$, $v $ has a nonneighbor in $N_{i-1}\}$
       where $i \ge 1$. Let $u \in N_{i-1}$ and $M_u$ denote the set of nonneighbors of $u$ in $N_i$.
       Since $\overline{G[T \cup S]}$ is connected, $N_i \neq \emptyset$.

       First, we show that $\alpha(G[M_u]) \le s-1$. Suppose not. Let $u_1,u_2,\ldots,u_s$ be $s$ pairwise nonadjacent vertices
       in $G[M_u]$. If $i=1$, let $u'\in V(G)\setminus{S\cup T}$ be the vertex such that $uu' \notin E(G)$ and $u'$ is complete to $T$.
        Then $\{u',u,u_1,u_2,\ldots,u_s\}$ induces a $K_{1,s}+P_1$. If $i \ge 2$, let $u'$ be a nonneighbor
       of $u$ in $N_{i-2}$. Then $\{u',u,u_1,u_2,\ldots,u_s\}$ induces a $K_{1,s}+P_1$. So $\alpha(G[M_u]) \le s-1$.

       Next, we show that each $N_i$ is bounded. Since $\omega(G[M_u]) \le \chi(G[T]) \le d$ and $\alpha(G[M_u]) \le s-1$, $|M_u|=|G[M_u]|
       \le R(s,d)$ by Ramsey's Theorem. Since $|N_0|=|S|\le c $ and $|N_i|\le \sum_{u \in N_{i-1}}|M_u|$, $N_i$ is bounded.

       Finally, we show that $i \le 2d +1$. Suppose not. By the definition of $N_i$, $N_i$ and $N_j$ are complete when $|i-j|>1$.
       Then we take a vertex $u_i \in N_i$ where $i$ is even. Now, $\{u_2,u_4,\ldots,u_{2\chi(G[T])+2}\}$ induces a $K_{\chi(G[T])+1}$, a contradiction. This shows $i \le 2\chi(G[T]) +1 \le 2d +1$.

       Therefore, $T$ is bounded.
       \end{proof}

       \begin{lemma}\label{lem:k3+2p1}
       Let $G$ be a $(P_m,\overline{K_3+2P_1})$-free graph and $c,d$ be two fixed integers. Let $S,T$ be two disjoint subsets of $V(G)$ such that $|S| \le c$, $\chi(G[T]) \le d$ and $G[T \cup S]$ is connected. If every vertex in $S$ is not adjacent to at least one of $x_1$ and
       $x_2$ where $x_1$ and $x_2$ are two arbitrary nonadjacent vertices in $T$ and there exists a vertex $w \in V(G)\setminus{(S\cup T)}$ such that $w$ is complete to $S \cup T$, then $T$ is bounded.
       \end{lemma}

       \begin{proof}
       Let $N_0 = S$, $N_i = \{v|v \in T\setminus {\cup_{j=0}^{i-1}N_j}$, $v $ has a neighbor in $N_{i-1}\}$
       where $i \ge 1$. Since $G$ is $P_m$-free, $i \le m-2$. Let $u \in N_{i-1}$ and $M_u$ denote the set of neighbors of $u$ in $N_i$.
       Since $G[T \cup S]$ is connected, $N_i \neq \emptyset$. In the following, we show that each $N_i$ is bounded.

      First, we consider the case of $i=1$. Since each vertex of $N_0 = S$ is not
      adjacent to at least one of $x_1$ and $x_2$ where $x_1,x_2$ are two arbitrary nonadjacent vertices in $N_1 \subseteq T$, $M_u$ is a clique. Then $|M_u| \le \chi(T) \le d$
      and $|N_1| \le \sum_{u \in N_0}|M_u| \le cd$. Next, we consider the case of $i \ge 2$. Let $u \in N_{i-1}$, and we claim that
      $\alpha(G[M_u]) \le 2$. Suppose not. Let $u_1,u_2,u_3$ be three pairwise nonadjacent vertices of $G[M_u]$. Then $\{u,w,u_1,u_2,u_3\}$
      induces a $\overline{K_3+2P_1}$. Note that $\omega(G[M_u]) \le d$. By Ramsey's Theorem, $|M_u|=|G[M_u]| \le R(3,d)$.
        Since $|N_i|\le \sum_{u \in N_{i-1}}|M_u|$ and $N_0$ is bounded, $N_i$ is bounded.

       Therefore, $T$ is bounded.
       \end{proof}

         A graph $G$ is {\em perfect} if $\chi(H)=\omega(H)$ for every induced subgraph $H$ of $G$. Another result we use is the famous Strong Perfect Graph Theorem.

       \begin{theorem}[The Strong Perfect Graph Theorem~\cite{CRST06}]\label{thm:SPGT}
		A graph is perfect if and only if it contains no odd holes or odd antiholes.
	\end{theorem}

The following theorem tells us there are finitely many 4-vertex-critical $P_5$-free graphs.

\begin{theorem}[\cite{BHS09,MM12}]\label{4-vertex-cri}
If $G = (V,E)$ is a 4-vertex-critical $P_5$-free graph, then $|V| \le 13$.
\end{theorem}

A property on $(P_5,claw)$-free graphs is shown as follows.

 \begin{lemma}\label{lem:p5,claw}
 Let $G$ be a connected $(P_5,claw)$-free graph and $c$ be a fixed integer.
 If $\chi(G) \le c$, then $G$ is bounded.
 \end{lemma}
 \begin{proof}
 Let $u \in V(G)$, $L_i =\{v|v\in V(G),\text{d}(u,v)=i\}$ where $i \ge 1$ and $L_0=\{u\}$.
  Let $v \in L_{i-1}$ and $M_v =\{w|wv \in E(G), w \in L_i\}$.
  Since $G$ is $(P_5,claw)$-free, $i \le 3$ and $\alpha(G[M_v]) \le 2$.
  Because $\omega(G) \le \chi(G) \le c$,
 and so $|M_v| \le R(3,c)$ by Ramsey's Theorem.
  Since $L_i \subseteq \cup_{v \in L_{i-1}}M_v$ and $L_0$ is bounded,
  $L_i$ is bounded. Therefore, $G$ is bounded.
 \end{proof}

 \section{The proof of \autoref{th:k-vertex-critical}}\label{sec:k1,4}

\begin{proof}
       We prove the theorem by induction on $k$. If $1 \le k \le 4$, there are finitely many $k$-vertex-critical $(P_5,K_{1,4}+P_1)$-free
       graphs by \autoref{4-vertex-cri}. In the following, we assume that $k \ge 5$ and that there are finitely many $i$-vertex-critical
       graphs for $i \le {k-1}$. Now, we consider the cases of $k$.

        Let $G=(V,E)$ be a $k$-vertex-critical $(P_5,K_{1,4}+P_1)$-free graph. We show that $|G|$ is bounded. Let $\mathcal{L} = \{K_k,\overline{C_{2k-1}}\}$. If $G$ has an induced subgraph isomorphic to
  a member $L \in \mathcal{L}$, then $|V(G)|=|V(L)|$ by the definition of vertex-criticality and so we are done.
  So, we assume in the following that $G$ has no induced subgraph isomorphic to a member in $\mathcal{L}$. Then $G$ is imperfect. Since $G$ is $k$-vertex-critical and $\chi(\overline{C_{2t+1}}) \geq k+1$ if $t \geq k$, it follows that $G$ does not contain $\overline{C_{2t+1}}$ for $t \geq k$.
  Moreover, since $G$ is $P_5$-free, it does not contain $C_{2t+1}$ for $t \geq 3$. It then follows from \autoref{thm:SPGT} that $G$ must contain some $\overline{C_{2t+1}}$ for $2 \le t \le k-2$. To finish the proof, we only need to prove the following lemma.

       \begin{lemma}\label{lem:2t+1antihole}
        If $G$ contains an induced $\overline{C_{2t+1}}$ for $2 \le t \le k-2$, then $G$ has finite order.
       \end{lemma}

       \begin{proof}
       Let $C = v_1,v_2,\ldots,v_{2t+1}$ be an induced $\overline{C_{2t+1}}$ such that $v_iv_j \in E(G)$ if and only if $|i - j| > 1$.
       All indices are modulo $2t+1$.  We partition $V(G)$ with respect to $C$.
       In the following, we shall write $S_2(v_i,v_{i+1})$ for $S_2(\{v_i,v_{i+1}\})$, $S_3(v_i,v_{i+1},v_{i+3})$ for $S_3(\{v_i,v_{i+1},v_{i+3}\})$, etc.

       \begin{claim}\label{st+1s2t}
       For every $X \subseteq C$ with $1 \le |X| \le 2t$, $|S(X)|\le R(4,k-2)$.
       \end{claim}
       \begin{proof}
       First, we show that $\alpha(G[S(X)]) \le 3$. Suppose not. Let $u_1,u_2,u_3,u_4$ be four
       pairwise nonadjacent vertices of $G[S(X)]$. Since $1 \le |X| \le 2t$, there exists an index $i \in \{1,2\ldots,2t+1\}$
       such that $v_i \notin N(S(X))$ but $v_{i+1} \in N(S(X))$. Then $\{u_1,u_2,u_3,u_4,v_i,v_{i+1}\}$
       induces a $K_{1,4}+P_1$. So $\alpha(G[S(X)]) \le 3$. Note that $\omega(G[S(X)]) \le k-2$. By Ramsey's Theorem,
       $|S(X)| = |G[S(X)]| \le R(4,k-2)$.
       \end{proof}

         By \autoref{st+1s2t},  $|S_{1} \cup \ldots \cup S_{2t}| \le R(4,k-2)(\tbinom{2t+1}{1} + \dots +\tbinom{2t+1}{2t})$.
        In the following, we bound $S_{0}$. First, we list some properties of $S_0$.

        \begin{claim}\label{s0:A}
        Let $A$ be a component of $S_0$, then each vertex of $ \cup_{i=1}^{2t}S_i$ is complete
        or anticomplete to $A$.
        \end{claim}

       \begin{proof}
       Let $uu'$ be an edge of $A$ and $v \in \cup_{i=1}^{2t}S_i$.
       Suppose that $v$ is mixed on $\{u,u'\}$. Since $v \in \cup_{i=1}^{2t}S_i$,
       there exist $v_i,v_j \in V(C)$ such that $v$ is mixed on $\{v_i,v_j\}$
       where $|i-j| \ge 2$. Then $\{v_i,v_j,v,u,u'\}$ induces a $P_5$.
       \end{proof}

       \begin{claim}
      Let $A$ be a component of $S_0$, then $A$ is bounded.
       \end{claim}

       \begin{proof}\label{s0:A is bounded}
        Let $u \in A$ and $L_i=\{v|v\in A, \text{d}(u,v)=i\}$ where $i \ge 0$. Since $G$ is $P_5$-free,
     $i \le 3$. Let $v \in L_{i}$ and $M_v =\{w|w \in L_{i+1}, wv \in E(G)\}$.
     Then $\alpha(G[M_v]) \le 3$.
     Suppose not. Let $u_1,\ldots,u_4$ be four pairwise nonadjacent
     vertices of $G[M_v]$. Then $\{v,u_1,u_2,u_3,u_4,v_i\}$ induces
     a $K_{1,4}+P_1$ where $v_i \in V(C)$. Since $\omega(G[M_v]) \le k-2$, $|M_v| = |G[M_v]| \le R(4,k-2)$.
     Note that  $L_i \subseteq \cup_{v \in L_{i-1}}M_v$,  then $|L_i| \le \sum_{v \in L_{i-1}}|M_v|$.
    Since $|L_0|=1$ is bounded, $L_i$ is bounded. Thus $A$ is bounded.
       \end{proof}

       \begin{claim}
       $S_0$ is bounded.
       \end{claim}
       \begin{proof}
     Let $L=\{u|u \in S_0 \text{~and~} u \text{ has a neighbor in } \cup_{i=1}^{2t}S_i\}$
     and $R = S_0 \setminus L$.

     First, we show that $L$ is bounded. Let $v \in  \cup_{i=1}^{2t}S_i$ and
     $M_v = \{u|uv \in E(G), u \in L\}$.  Since $G$ is $(P_5,K_{1,4}+P_1)$-free, we have $\alpha(G[M_v]) \le 3$.
     Since $\omega(G[M_v]) \le k-2$, $|M_v| = |G[M_v]| \le R(4,k-2)$.
     Since $L \subseteq \cup _{v\in \cup_{i=1}^{2t}S_i} M_v$, $|L| \le \sum_{v\in \cup_{i=1}^{2t}S_i}|M_v|$.
     Therefore, $L$ is bounded.

     Next, we show that $R$ is bounded.  Let $A$ be component
     of $R$, then $\chi(A) \le k-1$ and $A$ is bounded by \autoref{s0:A is bounded}.
     In the following, we show that the number of components of $R$
     is bounded.
     Let $A_1,A_2$ be two components of $R$.
     Then there exist $u_1 \in N(A_1)\setminus{N(A_2)}$ and
     $u_2 \in N(A_2)\setminus{N(A_1)}$ by \autoref{lem:XY}.
     We assume that $u_ia_i \in E(G)$ where $a_i \in A_i$ and $i \in \{1,2\}$.
     Then we have $u_1,u_2 \in S_{2t+1}$. So $u_1u_2 \in E(G)$,
     because otherwise $\{a_1,u_1,v_1,u_2,a_2\}$ induces a $P_5$.
      Let $A_3 \neq A_1,A_2$ be
     a component of $R$. Then $u_1 \notin N(A_3)$,
     because otherwise $\{a_1,a_2,a_3,u_1,v_1,v_2\}$ induces a $K_{1,4}+P_1$
     where $a_3 \in A_3$ and $a_3u_1 \in E(G)$. Similarly, $u_2 \notin N(A_3)$.
     So the number of components of $R$ is not more than $k-1$,
     otherwise there appears a $K_k$. Thus, $R$ is bounded.

     Therefore, $S_0$ is bounded.
       \end{proof}
       Finally, we bound $S_{2t+1}$.
       \begin{claim}
       $S_{2t+1}$ is bounded.
       \end{claim}

       \begin{proof}
      Let $N_0= S_0 \cup S_1 \cup \ldots \cup S_{2t}$, $N_i = \{v|v \in S_{2t+1}\setminus {\cup_{j=0}^{i-1}N_j}$, $v $ has a nonneighbor in $N_{i-1}\}$ where $i \ge 1$.

       If $\overline{G[S_0 \cup S_1 \cup \ldots \cup S_{2t+1}]}$ is connected. Since $S_0 \cup S_1 \cup \ldots \cup S_{2t}$ is bounded, $\chi(S_{2t+1}) \le k-(t+1)$ and for any $w \in S_0 \cup S_1 \cup \ldots \cup S_{2t}$ there exists
       a vertex $w' \in V(C)$ such that $ww' \notin E(G)$ and $w'$ is complete to $S_{2t+1}$, then $S_{2t+1}$ is bounded by \autoref{ST}.

       If $\overline{G[S_0 \cup S_1 \cup \ldots \cup S_{2t+1}]}$ is not connected. Then there exists an integer $j \ge 0$ such that $N_0,N_1,\ldots,N_j \neq \emptyset$ but $N_{j+1} = \emptyset$. Then $S_{2t+1} - \cup_{i=0}^{j}N_i$ is complete to $\cup_{i=0}^{j}N_i$, and so $S_{2t+1} - \cup_{i=0}^{j}N_i$
       is a homogeneous set.
       Since $\chi(S_{2t+1} - \cup_{i=0}^{j}N_i) \le k-(t+1)$, each component of $S_{2t+1} - \cup_{i=0}^{j}N_i$ is an $m$-vertex-critical graph with $1 \le m \le k-(t+1)$ by \autoref{lem:homoge}.
        By the inductive hypothesis, it follows that there are finitely many $m$-vertex-critical $(P_5,K_{1,4}+P_1)$-free graphs with $1 \le m \le k-(t+1)$. By the pigeonhole principle, the number of each kind of graph is not more than $1$. So, $S_{2t+1} - \cup_{i=0}^{j}N_i$ is bounded.
        For each $N_i$ with $1 \le i \le j$, $N_i$ is bounded by \autoref{ST}. Therefore, $S_{2t+1}$ is bounded.
\end{proof}

This completes the proof of \autoref{lem:2t+1antihole}.
  \end{proof}
  Therefore, \autoref{th:k-vertex-critical} holds.
  \end{proof}

  \section{The proof of \autoref{th:5-vertex-critical}}\label{sec:k1,s}

  \begin{proof}

  Let $G$ be a 5-vertex-critical $(P_5,K_{1,s}+P_1)$-free graph. We show that $|G|$ is bounded. Let $\mathcal{M} = \{K_5,\overline{C_{9}},\overline{C_{7}}+K_1,C_5+K_2\}$. If $G$ has an induced subgraph isomorphic to
  a member $M \in \mathcal{M}$, then $|V(G)|=|V(M)|$ by the definition of vertex-criticality and so we are done.
  So, we assume in the following that $G$ has no induced subgraph isomorphic to a member in $\mathcal{M}$. Then $G$ is imperfect. Since $G$ is $5$-vertex-critical and $\chi(\overline{C_{2t+1}}) \geq 6$ if $t \geq 5$, it follows that $G$ does not contain $\overline{C_{2t+1}}$ for $t \geq 5$.
  Moreover, since $G$ is $P_5$-free, it does not contain $C_{2t+1}$ for $t \geq 3$. It then follows from \autoref{thm:SPGT} that $G$ must contain $C_5$ or $\overline{C_7}$.

  Let $C=v_1,v_2,\ldots,v_{2t+1}$ be an induced $\overline{C_{2t+1}}$
  such that $v_iv_j \in E(G)$ if and only if $|i-j|>1$ where $t=2$ or 3.
  All indices are modulo $2t+1$.  We partition $V(G)$ with respect to $C$.
   Let $X \subseteq C$ with $1 \le |X| \le 2t$, then
  $\alpha(S(X)) \le s-1$. Suppose not. Let $u_1,u_2,\ldots,u_s$ be s pairwise
  nonadjacent vertices of $S(X)$. Since $1 \le |X| \le 2t$, there exists
  an index $i \in \{1,2,\ldots,2t+1\}$ such that $v_i \notin N(S(X))$
  but $v_{i+1}\in N(S(X))$. Then $\{v_i,v_{i+1},u_1,u_2,\ldots,u_s\}$ induces
  a $K_{1,s}+P_1$. Since $\omega(S(X)) \le k-2$, $|S(X)| \le R(s,k-2)$. Next,
  we bound $S_0$ and $S_{2t+1}$.

  Assume that $t=3$. Since $G$ is $\overline{C_{7}}+K_1$-free,
  $S_7 = \emptyset$. In this case, we only need to bound $S_0$.
  Let $A$ be a component of $S_0$, then $A$ is homogeneous
  by \autoref{s0:A}. Since $\chi(A) \le 3$, we have $A$ is $K_1$,
  $K_2$, $K_3$ or $C_5$ by \autoref{lem:homoge}.
  By the pigeonhole principle, the number of each kind of graph is not more than $1$.
  Thus, $S_0$ is bounded.

  Assume that $t=2$. Since $G$ is $C_5+K_2$-free,
  $S_5$ is stable. In the following, we first bound
  $S_0$. Let $L=\{u|u \in S_0 \text {~and~}u \text{ has a neighbor in } \cup_{i=1}^{4}S_i\}$
     and $R = S_0 \setminus L$. Let $v \in  \cup_{i=1}^{4}S_i$ and
     $M_v = \{u|uv \in E(G), u \in L\}$. Then $\alpha(G[M_v]) \le s-1$.
     Since $\omega(G[M_v]) \le k-2$, we have $|M_v| = |G[M_v]| \le R(s,k-2)$.
     Note that $L \subseteq \cup _{v\in \cup_{i=1}^{4}S_i}M_v$ and so
     $|L| \le \sum_{v \in \cup_{i=1}^{4}S_i}|M_v|$. Since $\cup_{i=1}^{4}S_i$ is bounded,
      $L$ is bounded.  Let $A_1$, $A_2$ be two
      components of $R$. Then there exist $u_1 \in N(A_1)\setminus{N(A_2)}$
     and $u_2 \in N(A_2)\setminus{N(A_1)}$ by \autoref{lem:XY}.  We have $u_1,u_2 \in S_5$.
     Then $\{a_1,u_1,v_1,u_2,a_2\}$ induces a $P_5$ where $a_i \in A_i$ and $i \in \{1,2\}$. So the number of
      components is not more than 1. By \autoref{s0:A is bounded}, each component
     of $R$ is bounded, so $R$ is bounded. Thus $S_0$ is bounded. Next, we bound $S_5$. Since $S_5$ is stable,  $|S_5| \le 2^{|\cup_{i=0}^4 S_i|}$ by the pigeonhole principle.

  This completes the proof of \autoref{th:5-vertex-critical}.
\end{proof}

  \section{The proof of \autoref{th:k3+2p1}}\label{sec:k3+2p1}
  \begin{proof}

   We prove the theorem by induction on $k$. If $1 \le k \le 4$, there are finitely many $k$-vertex-critical $(P_5,\overline{K_3+2P_1})$-free
       graphs by \autoref{4-vertex-cri}. In the following, we assume that $k \ge 5$ and that there are finitely many $i$-vertex-critical
       $(P_5,\overline{K_3+2P_1})$-free graphs for $i \le {k-1}$. Now, we consider the cases of $k$. 

        Let $G=(V,E)$ be a $k$-vertex-critical $(P_5,\overline{K_3+2P_1})$-free graph. We show that $|G|$ is bounded. Let $\mathcal{L} = \{K_k,\overline{C_{2k-1}}\}$. If $G$ has an induced subgraph isomorphic to
  a member $L \in \mathcal{L}$, then $|V(G)|=|V(L)|$ by the definition of vertex-criticality and so we are done.
  So, we assume in the following that $G$ has no induced subgraph isomorphic to a member in $\mathcal{L}$. Then $G$ is imperfect. Since $G$ is $k$-vertex-critical and $\chi(\overline{C_{2t+1}}) \geq k+1$ if $t \geq k$, it follows that $G$ does not contain $\overline{C_{2t+1}}$ for $t \geq k$.
  Moreover, since $G$ is $P_5$-free, it does not contain $C_{2t+1}$ for $t \geq 3$. It then follows from \autoref{thm:SPGT}, $G$ must contain some $\overline{C_{2t+1}}$ for $2 \le t \le k-2$. To finish the proof, we only need to prove the following lemma.

  \begin{lemma}\label{lem:2t+1{k3+2p1}}
        If $G$ contains an induced $\overline{C_{2t+1}}$ for $2 \le t \le k-2$, then $G$ has finite order.
       \end{lemma}

       \begin{proof}
       Let $C = v_1,v_2,\ldots,v_{2t+1}$ be an induced $\overline{C_{2t+1}}$ such that $v_iv_j \in E(G)$ if and only if $|i - j| > 1$.
       All indices are modulo $2t+1$.  We partition $V(G)$ with respect to $C$.

  \begin{claim}\label{s(x):bounded}
  For every $X \subseteq C$, if there exist $v_i,v_j \in X$ such that $|i-j| \ge 2$, then $S(X)$ is bounded.
  \end{claim}
  \begin{proof}
  First, we show that $\alpha(G[S(X)]) \le 2$. Suppose not. Let $u_1,u_2,u_3$ be three pairwise nonadjacent vertices of $G[S(X)])$. Then $\{u_1,u_2,u_3,v_i,v_j\}$
  induces a $\overline{K_3+2P_1}$. Note that $\omega(G[S(X)]) \le k-3$. By Ramsey's Theorem, $|S(X)|=|G[S(X)]| \le R(3,k-3)$.
  \end{proof}

   By \autoref{s(x):bounded}, it follows that $S(X)$ is bounded for $3 \le |X| \le 2t+1$. Since $G$ is $P_5$-free, $S_1=\emptyset$ and
  $S_2 = \cup_{i=1}^{2t+1}S(v_i,v_{i+1})$. In
   the following, we only need to bound $S_0$ and $S_2$.

   First, we bound $S_0$.

  \begin{claim}\label{s2-antic-s0}
  $S_0$ is anticomplete to $S_2$.
  \end{claim}
  \begin{proof}
  Let $u \in S_0$ and $v \in S(v_i,v_{i+1})$. If $uv \in E(G)$, then $\{u,v,v_i,v_{i-3},v_{i-1}\}$ is an induced $P_5$.
  \end{proof}

 \begin{claim}\label{s0sx}
 Let $A$ be a component of $S_0$, each vertex of $S(X)$  is complete or anticomplete to $A$ where $X \subseteq C$ and $3 \le |X| \le 2t$.
 \end{claim}
 \begin{proof}
 Let $uv$ be an edge of $S_0$ and $w \in S(X)$.  Since $3 \le |X| \le 2t$, there exist $v_i,v_j \in V(C)$ such that
  $w$ is mixed on $\{v_i,v_{j}\}$ and $|i-j| \ge 2$. If $uw \in E(G)$ and $vw \notin E(G)$, then $\{v,u,w,v_i,v_{j}\}$ is an induced $P_5$.
 \end{proof}

 \begin{claim}
 $S_0$ is bounded.
 \end{claim}
 \begin{proof}
 Let $L =\{u|u \in S_0, u \text{~ has a neighbor in~}{\cup_{i=3}^{2t}S_i}\}$ and $R = S_0 \setminus L$.
 Let $R_1 = \{u|u \in R, u\text {~ has a neighbor in~} S_{2t+1}\}$ and $R_2 = R \setminus R_1$.

 First, we show that $L$ is bounded.
 Let $A$ be a component of $L$, then each vertex of $\cup_{i=3}^{2t}S_i$ is complete or anticomplete to $A$ by \autoref{s0sx}.
 Then $A$ is $(P_5,claw)$-free and  so $A$ is bounded by \autoref{lem:p5,claw}. Next, we show that the number of components in $L$ is not more than
 $2^{|\cup_{i=3}^{2t+1}S_i|}$. Suppose not,
 then there are two components $A_1,A_2 \subseteq L$ having the same neighbors in $\cup_{i=3}^{2t+1}S_i$. Note that  $A_1$ and $A_2$ are anticomplete
 to $R$ by the definition of $R$  and combing with \autoref{s2-antic-s0},
 $A_1$ and $A_2$ have the same neighbors in $G$, which contradicts  with \autoref{lem:XY}. So $L$ is bounded.

 Next, we show that $R_1$ is bounded.
  Let $u \in S_{2t+1}$ and $M_u =\{v \in R_1|uv \in E(G)\}$. Let $B$ be a component
 of $M_u$, then $B$ is $(P_5,claw)$-free. So $B$ is bounded by \autoref{lem:p5,claw}. In the following, we show that the number of components in $M_u$ is not more than
 $2^{|S_{2t+1}|}$. Suppose not, then there exist two components $B_1,B_2$ having the same neighbors in $S_{2t+1}$. Since $B_1$ and $B_2$
 are not comparable, there exist $u_1 \in N(B_1)\setminus N(B_2)$ and $u_2 \in N(B_2)\setminus N(B_1)$.
 Assuming that $u_ib_i \in E(G)$ where $b_i \in B_i$ and $i \in \{1,2\}$.
 Since each component of $R_1$ is anticomplete to $L$ by the definition of $R$.
 Then $u_1,u_2 \in R$ and $uu_1, uu_2 \notin E(G)$. So $u_1u_2 \in E(G)$. Otherwise, $\{u_1,b_1,u,b_2,u_2\}$ induces a $P_5$. Then $\{u_2,u_1,b_1,u,v_1\}$ induces a $P_5$. So $M_u$ is bounded. Since $|R_1| \le \sum_{u\in S_{2t+1}}|M_u|$ and $S_{2t+1}$ is bounded, $R_1$ is bounded.

 Finally, we show that $R_2$ is bounded. Let $A$ be a component of $R_2$. By the connectivity of $G$, $N(A) \subseteq R_1$. We claim that each
 vertex of $R_1$ is complete or anticomplete to $A$. Suppose not. Let $uv$ be an edge of $A$, $w \in R_1$ and $w'$ be the
  neighbor of $w$ in $S_{2t+1}$. If $uw \in E(G)$ and $vw \notin E(G)$, then $\{v_1,w',w,u,v\}$ is an induced $P_5$. So $A$ is homogeneous.
  Since $\chi(A) \le \chi(S_0 ) \le k-1$, $A$ is an $m$-vertex-critical
$(P_5,\overline{K_3+2P_1})$-free graph with $1 \le m \le k-1$ by \autoref{lem:homoge}. By the inductive hypothesis, it
follows that there are finitely many $m$-vertex-critical $(P_5,\overline{K_3+2P_1})$-free graphs. By the pigeonhole principle,
the number of each kind of graph is not more than $2^{|R_1|}$. So $R_2$ is bounded.

 Therefore, $S_0$ is bounded.
 \end{proof}

 Next, we bound $S_2$.

 \begin{claim}\label{s(i,i+1):uv}
 Let $u,v \in S(v_i,v_{i+1})$ and $uv \notin E(G)$, each vertex of $S_{2t} \cup S_{2t+1}$ is not adjacent to at least one of $u$ and $v$.
 \end{claim}
 \begin{proof}
 Let $w \in S_{2t} \cup S_{2t+1}$. Then there exist $v_j \in C\setminus{\{v_i,v_{i+1}\}}$ such that $v_jw \in E(G)$, $v_jv_l \in E(G)$  and $wv_l \in E(G)$ where $v_l \in \{v_i,v_{i+1}\}$. If  $uw \in E(G)$ and $vw \in E(G)$, then
 $\{u,v,v_j,v_l,w\}$ is an induced $\overline{K_3+2P_1}$.
 \end{proof}

 \begin{claim}\label{claim:s(vi,vi+1)}
 For each $1 \le i \le 2t+1$, $S(v_i,v_{i+1})$ is complete to $S(C\setminus{\{v_i,v_{i+1}\}})$.
  \end{claim}
  \begin{proof}
  Let $u \in S(v_i,v_{i+1})$ and $v \in S(C\setminus{\{v_i,v_{i+1}\}})$. If $uv \notin E(G)$,
  then $\{u,v_{i+1},v_{i+3},v,v_{i+2}\}$ is an induced $P_5$.
  \end{proof}

 \begin{claim}\label{s(i,i+1):A is homo}
 Let $A$ be a component of $S(v_i,v_{i+1})$, each vertex of $\cup_{m=2}^{2t-1}S_m \setminus {[S(v_i,v_{i+1})\cup S(C\setminus{\{v_i,v_{i+1}\}})]}$ is complete or
 anticomplete to $A$.
 \end{claim}

\begin{proof}
\sloppy Let $uv$ be an edge of $S(v_i,v_{i+1})$ and $w \in \cup_{m=2}^{2t-1}S_m \setminus {[S(v_i,v_{i+1})\cup S(C\setminus{\{v_i,v_{i+1}\}})]}$. Assume that  $wu \in E(G)$ and $wv \notin E(G)$. Since $w \in \cup_{m=2}^{2t-1}S_m \setminus {[S(v_i,v_{i+1})\cup S(C\setminus{\{v_i,v_{i+1}\}})]}$ and $2 \le m \le 2t-1$, there exist $v_j,v_k \in C\setminus{\{v_i,v_{i+1}\}}$
such that $wv_j \in E(G)$, $v_jv_k \in E(G)$ and $wv_k \notin E(G)$. Then $\{v,u,w,v_j,v_k\}$ is an induced $P_5$.
\end{proof}

\begin{claim}
For each $1 \le i \le 2t+1$, $S(v_i,v_{i+1})$ is bounded.
\end{claim}

\begin{proof}
\sloppy Let $N_0 =S_{2t} \cup S_{2t+1}$, $N_j = \{v|v \in S(v_i,v_{i+1})\setminus{\cup_{l=0}^{j-1}N_l},
 v \text{~has a neighbor in~} N_{j-1}\}$
where $j \ge 1$. Let $u \in N_{j-1}$ and $M_u = \{w|w \in N_j, wu \in E(G)\}$.

If $G[S(v_i,v_{i+1}) \cup (S_{2t} \cup S_{2t+1})]$ is connected. Since $S_{2t} \cup S_{2t+1}$ is bounded and there exists a vertex $w \in \{v_i,v_{i+1}\}$ such that $w$ is complete to $S(v_i,v_{i+1}) \cup (S_{2t} \cup S_{2t+1})$. By \autoref{s(i,i+1):uv}, every vertex in $S_{2t} \cup S_{2t+1}$ is not adjacent to at least one of $u$
and $v$ where $u$ and $v$ are two arbitrary nonadjacent vertices in $S(v_i,v_{i+1})$. Note that $\chi(G[S(v_i,v_{i+1})]) \le k-1$. Then
$S(v_i,v_{i+1})$ is bounded by \autoref{lem:k3+2p1}.

If $G[S(v_i,v_{i+1}) \cup (S_{2t} \cup S_{2t+1})]$ is  not connected. Then there exists an integer $l \ge 0$ such that
$N_0,N_1,\ldots,N_l \neq \emptyset$ but $N_{l+1} = \emptyset$. Then $S(v_i,v_{i+1})- \cup_{j=0}^l N_j$ is anticomplete to
$\cup_{j=0}^l N_j$. Combined with \autoref{claim:s(vi,vi+1)} and \autoref{s(i,i+1):A is homo}, each component of $S(v_i,v_{i+1})- \cup_{j=0}^l N_j$ is a homogeneous set.
Since $\chi(S(v_i,v_{i+1}) - \cup_{j=0}^l N_j) \le k-1$, each component of $S(v_i,v_{i+1})- \cup_{j=0}^l N_j$ is an $m$-vertex-critical
$(P_5,\overline{K_3+2P_1})$-free graph with $1 \le m \le k-1$ by \autoref{lem:homoge}. By the inductive hypothesis, it
follows that there are finitely many $m$-vertex-critical $(P_5,\overline{K_3+2P_1})$-free graphs. By the pigeonhole principle,
the number of each kind of graph is not more than $2^{|\cup_{m=3}^{2t-1}S_m|}$. So $S(v_i,v_{i+1})- \cup_{j=0}^l N_j$ is bounded.
For each $N_j$ with $1 \le j \le l$, $N_j$ is bounded  by \autoref{lem:k3+2p1}. Therefore, $S(v_i,v_{i+1})$ is bounded.
\end{proof}
This completes the proof of \autoref{lem:2t+1{k3+2p1}}.
\end{proof}	
Thus, \autoref{th:k3+2p1} holds.
\end{proof}

\section{Characterization of all $5$-vertex-critical $(P_5,H)$-free graphs where $H \in \{K_{1,3}+P_1,K_{1,4}+P_1,\overline{K_3+2P_1}\}$}\label{sec:characterization}

From \autoref{th:k-vertex-critical}, \autoref{th:5-vertex-critical} and \autoref{th:k3+2p1}, we can already conclude that there are finitely many $5$-vertex-critical $(P_5,H)$-free graphs where $H \in \mathcal{S} := \{K_{1,3}+P_1,K_{1,4}+P_1,\overline{K_3+2P_1}\}.$ In the current section, however, we explicitly determine a list of all such graphs. This is necessary to explicitly construct a polynomial-time certifying algorithm for deciding whether a $(P_5,H)$-free graph $G$ is $4$-colorable for $H \in \mathcal{S}.$ More precisely, in a first phase one can use the polynomial-time algorithm from~\cite{HKLSS10} to determine whether $G$ is $4$-colorable and return the $4$-coloring produced by the algorithm as a certificate in case it is. If not, one of the $5$-vertex-critical graphs from the exhaustive list must occur as an induced subgraph. In the second phase, the algorithm can then iterate over each graph $G'$ in this list, check in polynomial time whether $G'$ occurs as an induced subgraph of $G$ and return it as a certificate in case it does.

In order to produce this finite list, we made two independent implementations of the algorithm described in~\cite{GS18}. We adapted this algorithm to be able to handle the $(P_5,H)$-free case for $H \in \mathcal{S}$ and made these implementations publicly available at~\cite{CodeJan} and~\cite{CodeJorik}. The algorithm recursively enumerates all $k$-vertex-critical $\mathcal{H}$-free graphs for a given positive integer $k$ and a set of graphs $\mathcal{H}$ (the pseudocode is given in \autoref{algo:recEnum}). More precisely, the algorithm additionally receives as an input a graph $I$ and in each recursion step one vertex is added as well all valid edges between this new vertex and the other vertices in all possible ways (note that $I$ occurs as an induced subgraph of the newly generated graphs). By calling the algorithm using the appropriate induced subgraphs $I$, one can enumerate all $k$-vertex-critical $\mathcal{H}$-free graphs (e.g.\ a correct but suboptimal choice would be $I=K_1$). In order to avoid having to add edges in all possible ways, the algorithm will use several pruning rules that forbid certain edge combinations with the guarantee that every $k$-vertex-critical $\mathcal{H}$-free graph was either generated before or has one of the newly generated graphs as an induced subgraph. For certain integers $k$ and sets of graphs $\mathcal{H}$, these pruning rules are powerful enough to let the algorithm terminate and hence produce an exhaustive list of $k$-vertex-critical $\mathcal{H}$-free graphs. A specific example of a pruning rule can be derived from noting that a vertex-critical graph cannot have two different vertices $u,v$ such that $N(u) \subseteq N(v)$ (i.e.\ comparable vertices, see \autoref{lem:XY} for a generalization of this fact). Hence, if $I$ is a graph that contains two such vertices $u$ and $v$ and occurs as an induced subgraph of a vertex-critical graph $G$, then there must be another vertex $w \in V(G)$ such that $w \in N_G(u)$, but $w \notin N_G(v)$. The algorithm can exploit this fact by adding the vertex $w$ to $I$ and a pruning rule which forbids that $u$ and $v$ are comparable vertices in the newly obtained graph. The algorithm employs multiple such pruning rules and we refer the interested reader to~\cite{GS18} for a complete overview of these rules. To test if a graph was not already generated before (cf.\ line~\ref{algo:isocheck} of \autoref{algo:recEnum}), we use the program \textit{nauty}~\cite{MP14} to compute a canonical form.

\begin{algorithm}[ht!b]
\caption{RecursivelyEnumerate(Integer $k$, Set of graphs $\mathcal{H}$, Induced graph $I$)}
\label{algo:recEnum}
\setstretch{0.9}
  \begin{algorithmic}[1]
		\STATE // This function recursively enumerates all $k$-vertex-critical $\mathcal{H}$-free graphs $G$ such that $I$ occurs as an induced subgraph of $G$
		\IF{$I$ was not previously enumerated AND $I$ is $\mathcal{H}$-free} \label{algo:isocheck}
			\IF{$\chi(I) \ge k$}
				\IF{$\chi(I-v)<k$ for all $v \in V(I)$}
					\STATE output $I$ // $I$ is $k$-vertex-critical $\mathcal{H}$-free
				\ENDIF	
			\ELSE
				\FOR{each graph $I'$ which is a one-vertex-extension of $I$}
					\IF{$I'$ is allowed by all pruning rules}
						\STATE RecursivelyEnumerate($k$,$\mathcal{H}$,$I'$)
					\ENDIF
				\ENDFOR					
			\ENDIF	
		\ENDIF	
  \end{algorithmic}
\end{algorithm}

\begin{table}[ht!b]
\footnotesize
\centering
\renewcommand{\arraystretch}{1.0}
\begin{tabular}{|c|*{6}{c|}}
\hline
 & \multicolumn{3}{c|}{5-vertex-critical $(P_5,H)$-free graphs} & \multicolumn{3}{c|}{5-critical $(P_5,H)$-free graphs} \\
\hline
Order & $\scriptstyle H=K_{1,3}+P_1$ & $\scriptstyle  H=K_{1,4}+P_1$ & $\scriptstyle  H=\overline{K_3+2P_1}$ & $\scriptstyle  H=K_{1,3}+P_1$ & $\scriptstyle  H=K_{1,4}+P_1$ & $\scriptstyle  H=\overline{K_3+2P_1}$ \\
\hline
$n=5$ & 1 & 1 & 1 & 1 & 1 & 1 \\
$n=6$ & 0 & 0 & 0 & 0 & 0 & 0 \\
$n=7$ & 1 & 1 & 1 & 1 & 1 & 1 \\
$n=8$ & 7 & 7 & 6 & 2 & 2 & 1 \\
$n=9$ & 198 & 199 & 180 & 8 & 9 & 8 \\
$n=10$ & 16 & 16 & 2 & 3 & 3 & 1 \\
$n=11$ & 24 & 24 & 5 & 7 & 7 & 2 \\
$n=12$ & 57 & 66 & 2 & 12 & 16 & 1 \\
$n=13$ & 40 & 67 & 5 & 8 & 19 & 2 \\
$n=14$ & 0 & 35 & 2 & 0 & 16 & 1 \\
$n=15$ & 0 & 71 & 3 & 0 & 13 & 1 \\
$n=16$ & 0 & 24 & 0 & 0 & 9 & 0 \\
$n=17$ & 0 & 23 & 4 & 0 & 7 & 2 \\
$n=18$ & 0 & 0 & 0 & 0 & 0 & 0 \\
$n=19$ & 0 & 0 & 1 & 0 & 0 & 1 \\
$n=20$ & 0 & 0 & 0 & 0 & 0 & 0 \\
$n=21$ & 0 & 0 & 1 & 0 & 0 & 1 \\
$n=22$ & 0 & 0 & 0 & 0 & 0 & 0 \\
$n=23$ & 0 & 0 & 1 & 0 & 0 & 1 \\
\hline
Total & 344 & 534 & 214 & 42 & 103 & 24 \\
\hline
\end{tabular}
\caption{An overview of the counts of all 5-vertex-critical and 5-critical $(P_5,H)$-free graphs for $H \in \{K_{1,3}+P_1,K_{1,4}+P_1,\overline{K_3+2P_1}\}.$}
\label{tab:counts}
\end{table}

This leads to the following characterization:
\begin{theorem}\label{th:characterization}
There are exactly 344 5-vertex-critical $(P_5,K_{1,3}+P_1)$-free graphs and the largest such graphs have order 13. There are exactly 534 5-vertex-critical $(P_5,K_{1,4}+P_1)$-free graphs and the largest such graphs have order 17. There are exactly 214 5-vertex-critical $(P_5,\overline{K_3+2P_1})$-free graphs and the largest such graphs have order 23.
\end{theorem}
\begin{proof}
It follows from the proofs of \autoref{th:k-vertex-critical}, \autoref{th:5-vertex-critical} and \autoref{th:k3+2p1} that every $5$-vertex-critical $P_5$-free graph is either isomorphic with $K_5$ or $\overline{C_9}$ or contains $\overline{C_5}$ or $\overline{C_7}$ as an induced subgraph. Hence, to produce the exhaustive list of $5$-vertex-critical $(P_5,H)$-free graphs for $H \in \mathcal{S}$, we have the graphs $K_5$ and $\overline{C_9}$ and we called \autoref{algo:recEnum} with the parameters $k=5$, $\mathcal{H}=\{P_5,H\}$ and $I=\overline{C_5}$ and $I=\overline{C_7}$.  In Table~\ref{tab:counts} we summarized for each order the number of $5$-vertex-critical $(P_5,H)$-free graphs with that order (as well as the number of the 5-critical $(P_5,H)$-free graphs). As expected, the two implementations of \autoref{algo:recEnum} (cf.~\cite{CodeJan} and~\cite{CodeJorik}) produced exactly the same graphs.
\end{proof}

We also made the exhaustive lists of the graphs from Table~\ref{tab:counts} publicly available on the \textit{House of Graphs}~\cite{CDG} at \url{https://houseofgraphs.org/meta-directory/critical-h-free}.

\section*{Acknowledgements}
\noindent The research of Jan Goedgebeur was supported by Internal Funds of KU Leuven. Jorik Jooken is supported by a Postdoctoral Fellowship of the Research Foundation Flanders (FWO) with grant number 1222524N. Shenwei Huang is supported by the Natural Science Foundation of China (NSFC) under Grant 12171256 and 12161141006.
We also acknowledge the support of the joint FWO-NSFC scientific mobility project with grant number VS01224N.

\end{document}